\documentclass[a4paper,11pt]{article}
\usepackage[english]{babel}

\usepackage{geometry}
\usepackage[applemac]{inputenc}  
\usepackage[T1]{fontenc}   
\usepackage{amsmath,amssymb,amsthm,mathrsfs}
\usepackage[all]{xy}
\usepackage{url}
\usepackage{enumerate}
\usepackage[pagebackref=true]{hyperref}
\usepackage{xspace}
\usepackage{authblk}

\usepackage{tikz}
\usetikzlibrary{matrix}
\usetikzlibrary{positioning}

\DeclareMathOperator{\Spec}{Spec}

\DeclareMathOperator{\Comm}{Comm}
\DeclareMathOperator{\Aut}{Aut}
\DeclareMathOperator{\Inn}{Inn}
\DeclareMathOperator{\Ker}{Ker}
\DeclareMathOperator{\Rad}{Rad}

\newcommand{\tdlc}{t.d.l.c.\xspace}
\newcommand{\tdlcsc}{t.d.l.c.s.c.\xspace}

\theoremstyle{definition}  \newtheorem{de}{Definition}[section]

\newtheorem{ex}[de]{Example}

\theoremstyle{plain}      
\newtheorem{theorem}[de]{Theorem}
\newtheorem{prop}[de]{Proposition}
\newtheorem{lem}[de]{Lemma}
\newtheorem{cor}[de]{Corollary}
\newtheorem{claim}{Claim}
\newtheorem{case}{Case} 

\theoremstyle{remark}      \newtheorem{rem}[de]{Remark}

\begin{document} 

\title{Totally disconnected locally compact groups\\
 with a linear open subgroup}
 \author[1]{Pierre-Emmanuel Caprace\thanks{F.R.S.-FNRS research associate, supported in part by the ERC (grant \#278469)}}
\author[1]{Thierry Stulemeijer\thanks{F.R.S.-FNRS research fellow}}

\affil[1]{\small{UCLouvain, 1348 Louvain-la-Neuve, Belgium}}

\date{September 25, 2014}

\maketitle

\begin{abstract}
We describe the global structure of totally disconnected locally compact groups having a linear open compact subgroup. Among the applications, we show that if a non-discrete, compactly generated, topologically simple,  totally disconnected locally compact group is locally linear, then it is a simple algebraic group over a local field. 
\end{abstract}

\tableofcontents

\section{Introduction}

A locally compact group is called \textbf{linear} if  it admits a continuous faithful finite-dimensional linear representation over a local field. It is called \textbf{locally linear} if it has an open subgroup which is linear. 
The goal of this paper is to study the class of totally disconnected locally compact groups (\tdlc groups for short) that are locally linear.  Roughly speaking, our main results ensure that such groups are built out of three kinds of elementary pieces: discrete groups, compact groups, and simple algebraic groups over local fields. In order to be more precise, let us   define a \textbf{topologically simple algebraic group over a local field} to be a locally compact group isomorphic to $H(k)/Z$, where $k$ is a local field, $H$ is an absolutely simple, simply connected, isotropic algebraic group over $k$, and $Z$ is the center of $H(k)$ (see \S \ref{sec:LAG} below for more details on those groups). We also say that a locally compact group is \textbf{locally solvable} (resp. \textbf{locally abelian}) if it has a solvable (resp. abelian) open subgroup. 

We can now state our main result.

\begin{theorem}\label{thm:Main1}
Let $ G $ be a \tdlc group having an open compact subgroup which is linear over a local field $k$. Then $ G $ has a series of closed normal  subgroups: 
$$ \lbrace 1\rbrace \unlhd R \unlhd G_{1} \unlhd G_{0} \unlhd G $$ 
enjoying the following properties. 


The group $ R $ is a closed characeristic subgroup and is locally solvable. The group $ G_{0} $ is an open characteristic subgroup of finite index in $ G $. The   quotient group $H_0 = G_0/R$, if non-trivial, has non-trivial closed normal subgroups, say $M_1, \dots, M_m$, satisfying the following properties.

\begin{enumerate}[(i)]

\item For some $l \leq m$ and all $i \leq l$, the group $M_i$ is a topologically simple algebraic group over a local field $k_i$, of the same characteristic and residue characteristic as $k$. In particular $M_i$ is compactly generated and abstractly simple. 

\item For all $j > l$, the group $M_j$ is   compact, h.j.i., and algebraic (in the sense of Definition~\ref{def:algebraicCompact}) over a local field $k_j$, of the same characteristic and residue characteristic as $k$.

\item Every non-trivial closed normal subgroup $N$ of $H_0$ contains $M_i$ for some $i \leq l$, or contains an open subgroup of $M_j$ for some $j > l$. 

\item The quotient group $H_{1} = G_1/R$ coincides with the product $  M_1 \dots M_m \cong M_1 \times \dots \times M_m$, which is closed in $H_0$.  In particular, $ H_{1} $ is compactly generated. Moreover  $H_{0}/H_{1} = G_0/G_{1}$ is  {locally abelian}.

\end{enumerate}
\end{theorem}

Notice the apparent analogy with the structure of general Lie groups, whose quotient by their solvable radical is semi-simple. It should however be emphasized that the characteristic subgroup $R$ afforded by Theorem~\ref{thm:Main1} is locally solvable, but not solvable in general: indeed, it can contain discrete normal subgroups of $G$ that are non-abelian free groups.

One nevertheless expects that the structure of the normal subgroup $R$ is not too mysterious. In order to make that statement precise, we define the class of \textbf{elementary groups} as the smallest class of \tdlc groups that contains all discrete groups, all profinite groups, and is closed under group extensions and directed unions of open subgroups. This class was first defined and investigated by Ph.~Wesolek in \cite{W14} in the second countable case, and then extended to the general \tdlc case in \cite{W14c}. We obtain the following consequence of Theorem~\ref{thm:Main1}. 

\begin{cor}\label{cor:Main2}
Let $G$ be a \tdlc group having a  compact open subgroup which is linear over a local field $k$. Then $ G $ has a series of closed characteristic subgroups
$$ \lbrace 1\rbrace\unlhd A\unlhd G_{1}\unlhd G_{0}\unlhd G $$
enjoying the following properties. 

The group $A$ is   elementary, $G_0$ is open of finite index in $G$,  and the quotient group $H_{0}=G_0/A$, if non-trivial, satisfies the following.
 
\begin{enumerate}[(i)]

\item $H_0$ has finitely many minimal closed normal subgroups, say $M_1, \dots, M_l$, and  every non-trivial closed normal subgroup of $H_0$ contains some $M_i$. 

\item Each $M_i$ is a topologically simple algebraic group over a local field $k_i$, of the same characteristic and residue characteristic as $k$. In particular $M_i$ is compactly generated and abstractly simple. 

\item The quotient group $H_1 = G_1/R$ coincides with the product $ M_1 \dots M_l \cong M_1 \times \dots \times M_l$, which is closed in $H_0$. Moreover the quotient $G_0/G_1 = H_0/H_1$ is  {locally abelian}; in particular it is   elementary.

\end{enumerate}

\end{cor}

We will see in due course that the characteristic subgroup $A$ afforded by Corollary~\ref{cor:Main2} contains, as expected, the subgroup $R$ afforded by Theorem~\ref{thm:Main1}.

Notice that when $k$ is of characteristic~$0$, the hypotheses of Theorem~\ref{thm:Main1} imply that $G$ is a $p$-adic Lie group, where $p$ is the residue characteristic of $k$. The conclusions of Corollary~\ref{cor:Main2} are then already known, due to Ph.~Wesolek: indeed, they follow from Corollary~1.5 in \cite{W14b}. Moreover, in that special case, the elementary quotient $G_0/G_1$ is even finite (the latter is however not true in positive characteristic, see Example~\ref{ex:Inseparable} below). The main novelty of our results is that they hold in \emph{all} characteristics. The key tool allowing for this uniformity is provided by the far-reaching results of R.~Pink \cite{P98} on compact subgroups of linear algebraic groups. 

Another special case of particular interest is when $G$ is assumed to be \textbf{topologically simple}, i.e. its only closed normal subgroups are the trivial ones. 

\begin{cor}\label{cor:TopolSimple}
Let $G$ be a  \tdlc group having a  compact open subgroup  which is linear over a local field $k$. If $G$ is topologically simple, then one of the following holds. 
\begin{enumerate}[(a)]
\item $G$ is discrete.

\item $G$ is non-discrete, not compactly generated, and locally solvable. 

\item $G$ is a topologically simple algebraic group over a local field $k'$, of the same characteristic and residue characteristic as $k$. In particular $G$ is compactly generated and abstractly simple. 

\end{enumerate}
\end{cor}

The following consequence is immediate. 

\begin{cor}\label{cor:S}
Let $G$ be a non-discrete, compactly generated, topologically simple, \tdlc{} group. If $G$ is locally linear, then $G$ is algebraic: indeed $G$ is a topologically simple algebraic group over a local field. 
\end{cor}

A systematic study of the class $\mathscr S$ of  non-discrete, compactly generated, topologically simple, \tdlc groups has been initiated in \cite{CRW14}. Corollary~\ref{cor:S} implies that the locally linear members of $\mathscr S$ are precisely the algebraic ones, and are thus all known since the latter algebraic groups have been classified by Kneser and Bruhat--Tits. 

Another, more specific, application of Corollary~\ref{cor:S} is the following. \emph{All those locally compact Kac--Moody groups (i.e. complete Kac--Moody groups over finite fields) that are known to be globally non-linear, are also not locally linear: none of their compact open subgroups is linear.} This applies to all irreducible Kac--Moody groups of non-spherical, non-affine type which are either of rank at least three (by \cite{CR}) or of rank two, and whose generalized Cartan matrix has $-1$ as an off-diagonal entry (by \cite{CR2}). 

Finally, we record the following application. 

\begin{cor}\label{cor:LC}
Let $G$ be a compactly generated, topologically simple, locally compact group. Then $G$ is linear over a (possibly Archimedean) local field if and only if $G$ belongs to one of the following classes:
\begin{itemize}
\item Finite simple groups.

\item Simple Lie groups. 

\item Simple algebraic groups over local fields. 
\end{itemize}

\end{cor}

\subsection*{Acknowledgement}

We thank Phillip Wesolek for his comments on an earlier version of this paper, and for pointing out that Corollary~\ref{cor:Main2} holds without any hypothesis of second countability.

\section{Algebraic \tdlc groups}

\subsection{Linear algebraic groups}\label{sec:LAG}

Let $ H $ be a \textbf{linear algebraic $ k $-group}, where $ k $ is a field. By definition, this means that $ H $ is a smooth affine group scheme of finite type over $ k $. Equivalently, $ H $ is (schematically) isomorphic to a smooth Zariski closed subgroup of $ GL_{n,k}$. As all algebraic groups used in this paper are linear, we omit this adjective in the sequel.

For $ \varphi \colon H\rightarrow H_{1} $ a morphism of algebraic $ k $-groups, we denote the evaluation of $ \varphi $ in a $ k $-algebra $ A $ by $ \varphi_{A} \colon H(A)\rightarrow H_{1}(A) $, or sometimes by $ \varphi \colon H(A)\rightarrow H_{1}(A) $.

When $ k $ is a Hausdorff (i.e. not anti-discrete) topological field, the group $H(k) $ inherits a Hausdorff topology, which does not depend on the embedding into $ GL_{n,k} $ (see e.g. \cite{PR94}). We adopt the convention that any topological statement will refer to that topology, and not to the Zariski topology, unless we explicitly add the prefix Zariski (e.g. Zariski-connected or Zariski-dense). When $k$ is a non-discrete locally compact field, the group $ H(k) $ is a locally compact second countable topological group.

A semisimple algebraic $ k $-group $ H $ is called \textbf{$ k $-simple} if it has no non-trivial Zariski-connected normal algebraic $ k $-subgroup. It is called \textbf{absolutely simple} if for any field extension $ k\rightarrow k' $, the algebraic $ k' $-group $ H\times_{k}k' $ is $ k' $-simple. Equivalently, $ H $ is absolutely simple if its root system is irreducible.

The study of Zariski-connected semisimple algebraic groups reduces for the most part to that of absolutely simple ones. Namely, a Zariski-connected, semisimple, simply connected (resp. adjoint) algebraic $ k $-group is the direct product of $ k $-simple, simply connected (resp. adjoint) algebraic $ k $-groups, and each factor is of the form $ \mathcal{R}_{k'/k}H $ for some absolutely simple, simply connected (resp. adjoint) algebraic $ k' $-group $ H $, where $ k' $ is a finite separable extension of $ k $. Here, $ \mathcal{R}_{k'/k}H $ denotes the Weil restriction of $ H $. For a proof of those facts, see e.g. Proposition 6.4.4, Remark 6.4.5 and Example 6.4.6 in \cite{C11}.

Let $ H $ be a Zariski-connected algebraic $ k $-group. As in \S6 of \cite{BT73}, we denote by $ H(k)^{+} $ the normal subgroup of $ H(k) $ generated by $ k $-rational points of split unipotent $ k $-subgroups of $ H $.

In order to properly understand the definition of an algebraic topologically simple \tdlc group, we need two results, that will also be invoked later on.

\begin{prop}\label{closed image}
Let $ k $ be a local field, let $ H $ be a simply connected, $ k $-simple algebraic $ k $-group. Any continuous homomorphism $ f \colon H(k)\rightarrow G $ to a locally compact group $G$ is a closed map.
\end{prop}

\begin{proof}
See Lemma 5.3 of \cite{BM96}. Note that the assumption there that the target of the map should be second countable is superfluous.
\end{proof}

\begin{theorem}\label{thmT}
Let $ k $ be a local field and let $ H $ be a $ k $-simple algebraic $ k$-group. Then any proper open subgroup of $ H(k)^{+} $ is compact.
\end{theorem}

\begin{proof}
See Theorem (T) of \cite{P82}.
\end{proof}

\begin{de}\label{algebraicinS}
Let $ G $ be a topologically simple \tdlc group. We say that $ G $ is \textbf{algebraic} if there exists a local field $ k $ and a $ k $-simple algebraic $ k $-group $ H $ such that $ G $ is topologically isomorphic to $ H(k)^{+}/Z(H(k)^{+}) $. 
\end{de} 

Note that by the main result of \cite{T64}, the quotient group $ H(k)^{+}/Z(H(k)^{+}) $ is abstractly simple. Moreover, the group  $ H(k)^{+}$ is compactly generated, as a consequence of Theorem~\ref{thmT}. 

The following theorem (which is a collection of results borrowed from \S6 of \cite{BT73}) shows that in the above definition, one can assume that $ H $ is simply connected, absolutely simple and that $ H(k)=H(k)^{+} $. This confirm that Definition~\ref{algebraicinS} is consistent with the definition of a topologically simple algebraic group over a local field given in the introduction. 

\begin{theorem}\label{thm:BT}
Let $k$ be a local field and $H$ be a  a $ k $-simple algebraic $ k $-group. 

\begin{enumerate}[(i)]
\item $ H(k)^{+} $ is trivial if and only if $ H $ is anisotropic over $ k $.

\item  Assume that  $ H $ is isotropic over $k$. Any central isogeny $ f \colon H\rightarrow H_{1} $ induces an isomorphism $ H(k)^{+}/Z(H(k)^{+}) \to H_{1}(k)^{+}/Z(H_{1}(k)^{+}) $.

\item If $ k' $ is a finite separable extension of $ k $ and $ H $ is  defined over $k' $, then $ (\mathcal{R}_{k'/k}H)(k)^{+}\simeq H(k')^{+} $.

\item  If $ H $ is isotropic, simply connected and absolutely simple, then we have $ H(k)=H(k)^{+} $.

\end{enumerate}
\end{theorem}
\begin{proof}
See \S6 of \cite{BT73}.
\end{proof}

\subsection{Hereditarily just-infinite groups}

An essential point in the proof of Theorem \ref{mainthm} is that the linearity of the open compact subgroup, say $U$, implies that $U$ has subquotients that are hereditarily just-infinite.

\begin{de}
A profinite group $ G $ is called \textbf{just-infinite} if it is infinite and every non-trivial closed normal subgroup of $ G $ is of finite index. A profinite group $ G $ is called \textbf{hereditarily just-infinite} (h.j.i.) if every open subgroup of $ G $ is just-infinite.

\end{de}

\begin{theorem}\label{h.j.i.}
Let $ U $ be an open compact subgroup of $ H(k) $, where $ H $ is a $ k $-simple, simply connected algebraic $ k $-group, and $ k $ is a local field of residue characteristic $ p $. Then $ U/Z(U) $ is a non-virtually abelian h.j.i. virtually pro-$ p $ group .

\end{theorem}

\begin{proof}
First note that $ U $ is Zariski-dense in $ H(k) $,  hence $ U $ is infinite and $ Z(U)=Z(H(k))\cap U $.  We want to show that if $ U_{1}\leq U/Z(U) $ is an open subgroup, then it is just-infinite. But the preimage of such a $ U_{1} $ is an open compact subgroup of $ H(k) $ as well, which has the same center than $ U $. Hence it suffices to prove that $ U/Z(U) $ is just-infinite and virtually pro-$p$. Now, this follows directly from the main result of \cite{R70}.

 It remains to show that $ U/Z(U) $ is not virtually abelian. But if it was, $ H(k) $ would have an open, hence Zariski-dense, abelian open subgroup, contradicting the fact that $ H $ is not abelian.
\end{proof}

We emphasize that when $\mathrm{char}(k) > 0$, the hypothesis that $ H $ is simply connected is essential in Theorem \ref{h.j.i.}. Indeed, as one immediately deduces from \cite{R70}, the above result does not hold if $ H $ is a $ k $-simple group whose universal cover $ \pi : \tilde{H}\rightarrow H $ is inseparable. Here is  an explicit example.

\begin{ex}\label{ex:Inseparable}
Consider the group $ H = PSL_{2} $ over the local field $ k = \mathbf{F}_{2}(\!(T)\!) $, and the open compact subgroup $ U = PSL_{2}(\mathbf{F}_{2}[\![T]\!]) $ in $ PSL_{2}(k) $. We insist that we consider $ PSL_{2} $ as the quotient scheme $ SL_{2}/\mu_{2} $ (over $ \Spec \mathbf{Z}$), and that $ PSL_{2}(\mathbf{F}_{2}[\![T]\!]) $ denotes the group of $ \mathbf{F}_{2}[\![T]\!] $ rational points of $ PSL_{2} $, not to be confused with the quotient group $ SL_{2}(\mathbf{F}_{2}[\![T]\!])/Z(SL_{2}(\mathbf{F}_{2}[\![T]\!])) $.

The universal cover is $ \pi : SL_{2}\rightarrow PSL_{2} $, which is purely inseparable over $ k $. Now, $ H(k)^{+} $ is a closed normal subgroup of $ H $ and is equal to $ \pi_{k}(SL_{2}(k)) $ (see \S6 of \cite{BT73}). Let us show explicitly that $ H(k)^{+}\cap U $ is not open in $ U $. It suffices to consider the sequence
\[
h_{i}=
  \begin{pmatrix}
    1+T^{\frac{1+2i}{2}} & 0\\
    0 & (1+T^{\frac{1+2i}{2}})^{-1} \\
  \end{pmatrix}
\]
whose elements are in $ H(k)\setminus H(k)^{+} $ and which converge to the identity.

Let us check that $ h_{i}\in H(k) $. If $ k[X_{11},X_{12},X_{21},X_{22}]/\det $ denotes the coordinate ring of $ SL_{2} $ in its standard coordinates, then the coordinate ring of $ PSL_{2} $ is the subring of $ k[SL_{2}] $ generated by all products $ X_{ij}X_{kl} $ where $ i,j,k,l \in \lbrace 1,2\rbrace $. This shows that $ h_{i} $ is indeed in $ H(k) $.
\end{ex}

\section{Locally normal subgroups}\label{sec:LocNorm}

The interaction between the local and global structure of general \tdlc groups has become more and more apparent in recent works. In this section, we review some tools from  \cite{CRW13} and \cite{W14} and use them to establish subsidiary facts that will be needed in the sequel.

\subsection{Locally $ C $-stable groups}

Let $G$ be a \tdlc group. 

\begin{de}\label{Cstable}
\begin{enumerate}[(1)]
\item The \textbf{quasi-center} of $G$, denoted by   $ QZ(G) $, is the characteristic subgroup of $ G $ consisting of all elements whose centraliser is open. More generally, given  $ H \leq  G $, we  define the \textbf{quasi-centraliser} of $ H $ in $ G $, denoted by $ QC_{G}(H) $, to be the subgroup of $ G $ consisting of those elements that centralise an open subgroup of $ H $.

\item A subgroup $ K\leq G $ is called \textbf{locally normal} if it is compact and normalised by an open subgroup of $ G $.

\item $ G $ is called \textbf{locally $ C$-stable} if $ QZ(G) $ is trivial and there is no non-trivial abelian locally normal subgroup.

\end{enumerate}

\end{de}

The following property of locally C-stable groups will be needed later. 

\begin{prop}\label{prop:locCstable}
Let $G$ be a locally C-stable \tdlc group. Then every locally normal subgroup of $G$ has trivial quasi-center. 
\end{prop}
\begin{proof}
See  \cite{CRW13}, Proposition~3.14.
\end{proof}

The locally $ C $-stable assumption is indeed a weakening of the hypotheses that $ G $ is a compactly generated and topologically simple, as  asserted by the following.

\begin{theorem}
Let $ G $ be a non-discrete topologically simple \tdlc which is compactly generated. Then $ G $ is locally $ C $-stable.
\end{theorem}

\begin{proof}
See \cite{CRW14}, Theorem 5.3.
\end{proof}

\subsection{The structure lattice}


Let $G$ be a \tdlc group. 

\begin{de}
Two subgroups $ H, K $ of $ G $ are \textbf{locally equivalent} if there exists a compact open subgroup $ U $ of $ G $ such that $ H\cap U=K\cap U $. The set of all local equivalence classes having a locally normal representative is called the \textbf{structure lattice} of $ G $, and is denoted by $ \mathcal{LN}(G) $. 
\end{de}

$ \mathcal{LN}(G) $ is a lattice in a natural way:  the meet operation is the intersection (of any representatives) and the join operation is the product (of well-chosen representatives). Obviously, $ [\lbrace e\rbrace] $, the local equivalence class of the trivial subgroup, is the minimum of $ \mathcal{LN}(G) $ and we denote it by $ 0 $. At the other extreme, the local equivalence class of compact open subgroups of $ G $ is the maximum of $ \mathcal{LN}(G) $ and we denote it by $ \infty $. We refer the reader to Section 2 of \cite{CRW13} for a more detailed discussion of $ \mathcal{LN}(G) $.

An \textbf{atom} of $ \mathcal{LN}(G) $ is a minimal non-zero element. 
The following lemma is a first elementary observation about the role of h.j.i. locally normal groups in the structure lattice.

\begin{lem}\label{h.j.i. structure lattice}
Let $ G $ be a \tdlc group and  $ \alpha \in \mathcal{LN}(G) $ have a locally normal representative $ V $ which is h.j.i. Then $ \alpha $ is an atom of $ \mathcal{LN}(G) $.
\end{lem}

\begin{proof}
First note that by definition, $ V $ is not discrete, hence $ \alpha \neq 0 $.

Let $ \beta \in \mathcal{LN}(G) $, and assume that $ \beta \leq \alpha $. We want to show that either  $ \beta = 0 $ or $ \beta = \alpha $. Let $ L $ be a locally normal representative of $ \beta $ (i.e. $ \beta = [L] $), so that our assumption translates as $ [L\cap V]=[L] $.

Consider an open compact subgroup $ W $ of $ N_{G}(L)\cap N_{G}(V) $. Now, $ L\cap W\cap V $ is normal closed in $ W $, hence also in $ W\cap V $. But since the latter is just-infinite, we obtain that $  L\cap W\cap V $ is either trivial, or open in $ W\cap V $, as wanted.
\end{proof}

When $ G $ is a product of h.j.i. groups, we can refine the previous lemma as follows.

\begin{lem}\label{2m}
Let $ G\simeq V_{1}\times \dots \times V_{m} $ be a profinite group which splits as the direct product of finitely many  h.j.i. closed subgroups, none of which is virtually abelian.

Then the structure lattice $ \mathcal{LN}(G) $ is isomorphic to the Boolean algebra of all subsets of $ \lbrace 1,\dots,m\rbrace $. Moreover, for a locally normal representative $ K $ of an element of $ \mathcal{LN}(G) $, there exist $ i_{1},\dots,i_{k}\in \lbrace 1,\dots,m\rbrace $ such that $ U_{i_{1}}\dots U_{i_{k}} \leq K\leq V_{i_{1}}\dots V_{i_{k}} $, where $ U_{i_{j}} $ is an open subgroup of $ V_{i_{j}} $ for all $ j $.

\end{lem}

\begin{proof}
Let $ \alpha_{i}=[V_{i}] $ be the local equivalence class having representative $ V_{i} $. Obviously, $ \lbrace \alpha_{i_{1}}\vee \dots \vee \alpha_{i_{k}}$  $\vert$  $ 1\leq i_{1}<\dots<i_{k}\leq m$,  $k=1,\dots,m \rbrace $ are different elements in $ \mathcal{LN}(G) $. We want to show that all elements of $ \mathcal{LN}(G) $ are of this form.

Let $ \beta \in \mathcal{LN}(G) $, let $ K $ be a locally normal representative of $ \beta $, 
and consider the projection $ \pi_{i}(K) $ onto the $i$-th factor $V_i$. Since $ K $ is normalised by an open subgroup of $ G $, it is  normalised by an open subgroup of  $ V_{i} $. Furthermore, $ \pi_{i}(K) $ is compact, hence closed.  But $ V_{i} $ is h.j.i., so that $ \pi_{i}(K) $ is either trivial or open in $ V_{i} $. Reordering the $ V_{i} $'s if necessary, we can assume that $ \pi_{i}(K) $ is open in $ V_{i} $ for $ i\leq k $, and $ \pi_{i}(K) $ is trivial for $ i>k $.

We obviously have that $ [K]\leq [V_{1}]\vee \dots \vee [V_{k}] $, and we now prove the reverse inequality. For this, consider $ K\cap V_{i} $, for $ i\leq k $ (where $ V_{i} $ is seen as a subgroup of $ G $ via the natural injection). It is a locally normal subgroup of $ V_{i} $, hence it is either trivial or open in $ V_{i} $. But if $ K\cap V_{i} $ were trivial, we would have $ [K,N_{V_{i}}(K)]\subseteq K\cap V_{i}=\lbrace e\rbrace $, contradicting $ \lbrace e\rbrace\neq [\pi_{i}(K),\pi_{i}(K)\cap N_{V_{i}}(K)]\subseteq [K,N_{V_{i}}(K)] $. We conclude that for $ i\leq k $, $ [V_{i}\cap K]=[V_{i}] $, so that $ [V_{1}]\vee \dots \vee [V_{k}] = [(N_{V_{1}}(K)\cap K)\dots (N_{V_{k}}(K)\cap K)]\leq K $, as wanted.

For the last assertion, note that if $ K $ is a locally normal representative of an element in $ \mathcal{LN}(G) $, then the second paragraph of this proof shows that $ \pi_{i}(K) $ is open in $ V_{i} $ for $ i \in \lbrace i_{1},\dots,i_{k}\rbrace $ and is trivial otherwise, so that $ K\leq V_{i_{1}}\dots V_{i_{k}} $. But the third paragraph shows then that $ N_{V_{i_{j}}}(K)\cap K $ is open in $ V_{i_{j}} $ for all $ j $, and that $ (N_{V_{i_{1}}}(K)\cap K)\dots(N_{V_{i_{k}}}(K)\cap K)\leq K $, as wanted.
\end{proof}

We now consider profinite groups that are virtually a direct product of h.j.i. groups. 

\begin{lem}\label{lem:ProductHJI}
Let $G$ be a profinite group without non-trivial finite normal subgroup, and having an open subgroup $G_0 \cong V_{1}\times \dots \times V_{m} $ which splits  as the direct product of finitely many  h.j.i. closed subgroups, none of which is virtually abelian.

Then $G$ has a characteristic open subgroup $ G_1 \simeq W_1\times \dots \times W_{m} $, contained in $G_0$, which splits  as the direct product of finitely many  h.j.i. closed subgroups, where $W_i $ is an open subgroup of $V_i$.
\end{lem}

\begin{proof}
We first claim that $G$ is locally C-stable. Indeed, every h.j.i. group which is not virtually abelian has trivial quasi-center, by Proposition~5.1 from \cite{BEW11}. Therefore, the quasi-center of $G$ must be finite, hence trivial by hypothesis. The fact that the only abelian locally normal subgroup of $G$ is the trivial one follows from Lemma~\ref{2m}. The claim stands proven. 

Let now $\alpha_i = [V_i] \in \mathcal{LN}(G)$. Since the quasi-centraliser $QC_G(V_i)$ depends only on the local class $\alpha_i$, we denote it by  $QC_G(\alpha_i)$, following the convention adopted in \cite{CRW13}. Next we set $L_i = QC_G(QC_G(\alpha_i))$. Since $G$ is locally C-stable, we may invoke  Lemma~3.11(ii) from \cite{CRW13}, which shows that $L_i = C_G(C_G(V_i))$. Thus, the subgroup $L_i$ is closed in $G$.  By Lemma~\ref{2m}, the automorphism group $\mathrm{Aut}(G)$ permutes the $\alpha_i$ and, hence, permutes the closed subgroups $L_i$. 

We now claim that $[L_i] = \alpha_i$ and that $L_i$ commutes with $L_j$ for all $i \neq j$. Note that $ V_{i}\subseteq L_{i} = C_G(C_G(V_i)) $, hence we just have to show that $ V_{i} $ is open in $ L_{i} $. Set $ P_{i} := \prod \limits_{j\neq i}V_{j} $. Then $ G\geqslant L_{i}P_{i}\geqslant V_{i}P_{i} = G_{0} $. But $ G_{0} $ is open in $ G $, so that the index of $ V_{i} $ in $ L_{i} $ is finite, as wanted. For the second assertion of the claim, starting from $ V_{i}\subseteq C_{G}(V_{j}) $ (which is true for all $i \neq j$), we get $ L_{i} = C_G(C_G(V_i)) \subseteq C_G(C_G(C_{G}(V_{j})))  = C_{G}(V_{j}) $. So that finally, $ C_{G}(L_{i})\supseteq C_G(C_G(V_{j})) = L_{j} $, as was to be shown.

Therefore the subgroup $G_2 \leq G$ generated by the $L_i$'s is characteristic, open, and isomorphic to the direct product $L_1 \times \dots \times L_m$. Since each $L_i$ is h.j.i., it has a basis of identity neighbourhoods consisting of characteristic subgroups. Therefore, the same is true for $G_2$, and therefore $G_2$ has a characteristic subgroup $G_1$ contained in $G_0$, which has the desired form. 
\end{proof}

Building upon Lemmas~\ref{h.j.i. structure lattice} and~\ref{2m}, we obtain the following more technical fact.

\begin{lem}\label{structurelattice}
Let $G$ be a profinite group having a closed normal subgroup $ \Gamma \simeq V_{1}\times \dots \times V_{m} $ which splits  as the direct product of finitely many  h.j.i. closed subgroups, none of which is virtually abelian. Assume further that $ G $ contains no non-trivial abelian locally normal subgroup, and that $ G/\Gamma $ is abelian. 

Then for all $ 0\neq \beta \in \mathcal{LN}(G) $, there exists $ i \in \lbrace 1,\dots,m\rbrace $ such that $ [V_{i}]\leq \beta $. In particular, the set of all the atoms of $ \mathcal{LN}(G) $ is precisely $ \lbrace [V_{1}],\dots,[V_{m}] \rbrace $.
\end{lem}

\begin{proof}
We first have to check that $ V_{i} $ has an open normaliser in $ G $. But $ G $ acts on the atoms of $ \mathcal{LN}(\Gamma ) $, which is a finite set by Lemma \ref{2m}. Also, in view of the last assertion of that lemma and the fact that $ \Gamma $ is normal, $ [gV_{i}g^{-1}]=[V_{i}] $ if and only if $ g\in N_{G}(V_{i}) $. Hence, $ N_{G}(V_{i}) $ is of finite index in $ G $. As it is also closed, because $ V_{i} $ is, we conclude that it is open, as wanted.

Now let $ \beta $ be a non-trivial element of $ \mathcal{LN}(G) $ and set $ \alpha_{i} = [V_{i}] $. Recall that by Lemma~\ref{h.j.i. structure lattice}, the $ \alpha_{i} $'s are atoms in $ \mathcal{LN}(G) $. We want to show that $ \beta \geq \alpha_{i} $ for some $ i $, and we now separate the proof into two cases.

\begin{case}
$ \beta \wedge (\alpha_{1}\vee \dots \vee \alpha_{m})=0 $.
\end{case}

As is usual in this situation, we can find $ W $ open compact in $ G $ and some locally normal representative $ K $ (resp. $ U_{i} $) of $ \beta $ (resp. $ \alpha_{i} $) such that $ K $ and the $ U_{i} $ are normal in $ W $. Now, the assumption translates as $ K\cap (U_{1}U_{2}\dots U_{m}) = F $, a finite subgroup. Shrinking $ K $ again if necessary, we may assume that $ K\cap (U_{1}U_{2}\dots U_{m}) = \lbrace e\rbrace $. We may also assume that $U_i \leq V_i$.

Hence, we have continuous injective maps $ K\rightarrow W/(U_{1}U_{2}\dots U_{m})\rightarrow G/\Gamma $. But since the latter is abelian, so is $ K $. In view of the hypothesis, we conclude that $ \beta = 0 $, a contradiction.

\begin{case}
$ \beta \wedge (\alpha_{1}\vee \dots \vee \alpha_{m})\neq 0 $.
\end{case}

Using Lemma \ref{2m}, we conclude that for some $ i $, we have $ \alpha_{i}\leq \beta \wedge (\alpha_{1}\vee \dots \vee \alpha_{m})\leq \beta $, as wanted.
\end{proof}

\subsection{Radical theories}\label{radicaltheories}

In  this section, we review  the definition and basic properties of two characteristic subgroups of a general \tdlc group. The first is the $ [A] $-regular radical $ R_{[A]}(G) $, defined  in \cite{CRW13}, and the second is the elementary radical $ R_{\mathcal{E}}(G) $, defined in \cite{W14} under the assumption that $G$ is second countable. 

\begin{de}
Let $[A]$ be the smallest class of profinite groups, stable under isomorphism, such that the following conditions hold:
\begin{enumerate}[(a)]
\item  $[A]$  contains all abelian profinite groups and all finite simple groups.
\item If $ U\in [A] $ and $ K $ is a closed normal subgroup of $ U $,  then $ K\in [A] $ and $ U/K\in [A] $.
\item Given a profinite group $ U $ that is a (not necessarily direct) product of finitely many closed normal subgroups belonging to $ [A] $, then $ U\in [A] $.
\end{enumerate}

Given a profinite group $U$, a subgroup $K$ of $U$ is called \textbf{$[A]$-regular} in $U$ if for every closed normal subgroup $L$ of $U$ not containing $K$, the image of $K$ in the quotient $U/L$ contains a non-trivial locally normal subgroup of $U/L$ belonging to the class $[A]$.  Given a \tdlc group $G$ and a closed subgroup $H$, we say that $H$ is \textbf{$[A]$-regular} in $G$ if $H \cap U$ is  {$[A]$-regular} in $U$ for all open compact subgroups $U$ of $G$. 
\end{de}

The \textbf{$[A] $-regular radical} of a \tdlc group $G$ is defined to be the characteristic subgroup identified by the following result.

\begin{theorem}\label{A-rad}
Let $G$ be a \tdlc group. Then $G$ has a closed characteristic subgroup $ R_{[A]}(G) $, which is characterised by either of the following properties:
\begin{enumerate}[(i)]
\item $R_{[A]}(G) $ is the largest subgroup of $G$ that is $[A]$-regular. 

\item $R_{[A]}(G) $  is the  smallest closed normal subgroup $ N $ such that $ G/N$ is locally C-stable.

\end{enumerate}
\end{theorem}

\begin{proof}
See Theorem~6.10 in \cite{CRW13}.
\end{proof}
 
We now move on to elementary groups following \cite{W14} and \cite{W14c}. We first restrict to \tdlc groups that are second countable (\tdlcsc for short).

\begin{de}
The class $ \mathcal{E}_{\mathrm{sc}}$ is defined as the smallest class of \tdlcsc groups such that
\begin{enumerate}[(a)]
\item $ \mathcal{E}_{\mathrm{sc}}$ contains all metrisable profinite groups and all countable discrete groups.
\item $ \mathcal{E}_{\mathrm{sc}}$ is closed under taking group extensions. 
\item $\mathcal{E}_{\mathrm{sc}} $ is closed under countable directed unions of open subgroups.
\end{enumerate}
\end{de}

A key feature, due to Ph.~Wesolek, is the existence of a radical belonging to the class $\mathcal{E}_{\mathrm{sc}}$, asserted in the following.

\begin{prop}\label{elementaryradicalexistence}
Let $G$ be a \tdlcsc group. Then $G$ has a largest   closed normal subgroup $\Rad_{\mathcal{E}_{\mathrm{sc}}}(G) $ which belongs to the class 
$\mathcal{E}_{\mathrm{sc}}$. 
\end{prop}
\begin{proof}
See  \cite{W14}, Proposition 7.4.
\end{proof}

The relation between the two radicals introduced above is elucidated by the following. 

\begin{prop}\label{thm:radicals}
Let $ G $ be a \tdlcsc group. Then $  R_{[A]}(G)\leq \Rad_{\mathcal{E}_{\mathrm{sc}}}(G) $. In particular, $ G/\Rad_{\mathcal{E}_{\mathrm{sc}}}(G) $ is locally C-stable, and $ R_{[A]}(G) $belongs to   $\mathcal{E}_{\mathrm{sc}}$. 
\end{prop}

\begin{proof}
See \cite{W14}, Corollary 9.12 and 9.13.
\end{proof}

We now briefly explain how one can drop the second countability assumption, following \cite{W14c}. This discussion was suggested to us by Ph.~Wesolek.  

\begin{de}
The class of \textbf{elementary groups} is the smallest class $ \mathcal{E}$ of \tdlc groups such that
\begin{enumerate}[(a)]
\item $ \mathcal{E}$ contains all profinite groups and all discrete groups.
\item $ \mathcal{E}$ is closed under taking group extensions.
\item $ \mathcal{E}$ is closed under directed unions of open subgroups.
\end{enumerate}
\end{de}

It should be stressed that our choice of terminology is slightly different from Wesolek's: what he called the class of \emph{elementary groups} and denoted by $\mathcal E$  in the references \cite{W14} and \cite{W14c}  and   what is denoted by $\mathcal{E}_{\mathrm{sc}}$ in the present paper. Moreover, the class denoted here by $\mathcal E$ is denoted by $\mathcal E^\ast $ in  \cite{W14c}. We believe that our choice is   natural in the present context, and should not cause any confusion.

The inclusion $\mathcal{E}_{\mathrm{sc}}  \subset \mathcal{E} $ is clear from the definitions. Conversely, we have the following.

\begin{lem}\label{lem:reducingtos.c.case}
Let $ G\in \mathcal{E} $. If $ G $ is secound countable, then $ G\in \mathcal{E}_{\mathrm{sc}} $.
\end{lem}

\begin{proof}
This is a particular case of \cite{W14c}, Proposition 4.3.
\end{proof}

This lemma allows us to deduce that many properties of $\mathcal{E}_{\mathrm{sc}} $ generalise to $\mathcal{E} $, as follows (see also Theorem~\ref{thm:starradicals} below).

\begin{prop}\label{prop:transferfroms.c.togeneral}
Let $G$ be a \tdlc group. 
\begin{enumerate}[(i)]
\item Let $ H $ be a dense normal subgroup of   $ G $ such that $ H\in \mathcal{E}  $. Then $ G\in \mathcal{E} $.
\item If $ G $ is   locally solvable, then $ G  \in \mathcal{E}$.
\end{enumerate}
\end{prop}

\begin{proof}
Write $ G = \bigcup \limits_{i\in I} O_{i} $ as a directed union of compactly generated   open subgroups. By \cite{KK44}, for each $ i \in I $, there exists a compact normal subgroup $ K_{i}\leq O_{i} $ such that $ O_{i}/K_{i} $ is a \tdlcsc group.

\begin{enumerate}[(i)]
\item  Let now $ H $ be a dense normal subgroup of   $ G $ such that $ H\in \mathcal{E}$. Then $ (O_{i}\cap H)K_{i}/K_{i} $ is a dense normal subgroup of  $ O_{i}/K_{i} $. Moreover, by Lemma~\ref{lem:reducingtos.c.case}, it belongs to $ \mathcal{E}_{\mathrm{sc}}$. In the second countable case, the desired result is known, namely  $ O_{i}/K_{i} \in \mathcal{E}_{\mathrm{sc}}$ by  \cite[Theorem 1.4]{W14}. Hence, for each $ i $, the group $ O_{i} $ is compact-by-elementary, hence  elementary. Therefore $ G $ is  itself   elementary, as required. 

\item For each $i$, the groups $O_i$ and $O_i/K_i$ are locally solvable.  By \cite[Theorem 8.1]{W14} we have  $O_i/K_i \in \mathcal{E}_{\mathrm{sc}}$. Therefore $O_i$ is compact-by-elementary, and we conclude as in the proof of (i). \qedhere
\end{enumerate}
\end{proof}



The \textbf{elementary radical} is defined to be the characteristic subgroup identified by the following result, which is a straightforward adaptation of \cite[Proposition 7.4]{W14}. 

\begin{theorem}\label{starelementaryradicalexistence}
Let $G$ be a \tdlc group. Then $G$ has a largest closed normal subgroup $\Rad_{\mathcal{E}}(G) $ which is   elementary.
\end{theorem}

\begin{proof}
Let $N_i$ be an ascending chain of closed normal subgroups of $G$ that are elementary. Let $U< G$ be a compact open subgroup. Then $N_i U$ is elementary for each $i$, hence so is the union $O = \bigcup_i N_i U$. It follows that $\overline{\bigcup_i N_i } \leq O$ is elementary, since $\mathcal E$ is closed under taking closed subgroups in view of  \cite[Theorem~4.6(b)]{W14c}. It follows from Zorn's lemma that the collection of elementary closed normal subgroups of $G$ has maximal elements. In fact there is a unique such,  since the closure of the product of any two of them is itself elementary by  Proposition~\ref{prop:transferfroms.c.togeneral}(i). 
\end{proof}

Finally, we extend Theorem~\ref{thm:radicals} to the general case. 

\begin{theorem}\label{thm:starradicals}
Let $ G $ be a \tdlc group. Then $  R_{[A]}(G)\leq \Rad_{\mathcal{E}}(G) $. In particular, $ G/\Rad_{\mathcal{E}}(G) $ is locally C-stable, and $ R_{[A]}(G) $ is elementary.
\end{theorem}

\begin{proof}
Write $ G = \bigcup \limits_{i\in I} O_{i} $ as a directed union of compactly generated   open subgroups. By \cite{KK44}, for each $ i \in I $, there exists a compact normal subgroup $ K_{i}\leq O_{i} $ such that $ O_{i}/K_{i} $ is a \tdlcsc group.

For each $i \in I$, we have $ O_{i}\cap R_{[A]}(G) = R_{[A]}(O_{i}) $ in view of  \cite[Proposition 6.14(ii)]{CRW13}. Let $ \pi_{i}\colon O_{i}\rightarrow O_{i}/K_{i} $ be the projection. Since $ [A] $-regularity is stable under quotients by closed normal subgroups, we have   $ R_{[A]}(O_{i})\leq \pi_{i}^{-1}(R_{[A]}(O_{i}/K_{i})) $. Notice that $R_{[A]}(O_{i}/K_{i}) $ is elementary   by Theorem~\ref{thm:radicals}, hence so is  $ \pi_{i}^{-1}(R_{[A]}(O_{i}/K_{i}))$, and thus also $ R_{[A]}(O_{i})$ by  \cite[Theorem~4.6(b)]{W14c}. We conclude that $ R_{[A]}(G) $ is a directed union of open  elementary subgroups, and is thus elementary, as required.
\end{proof}

\section{Compact subgroups of linear algebraic groups}

As outlined in the introduction, Theorem \ref{mainthm} relies essentially on the results obtained by Pink in \cite{P98}. The goal of this section is to review those results and to adapt them to our needs.

\subsection{The group of abstract commensurators}

Another important object used in the local-to-global transfer lying behind our main results is the group of abstract commensurators of a profinite group, first defined and investigated by Barnea--Ershov--Weigel in \cite{BEW11}.

\begin{de}
Let $ U $ be a profinite group. The \textbf{group of abstract commensurators} of $ U $, denoted $ \Comm(U) $, is defined as follows. Consider the set $ E $ of isomorphisms $ \alpha : U_{1}\rightarrow U_{2} $, where the $  U_{i} $'s are open compact subgroups of $ U $ and $ \alpha $ is a topological isomorphism. Define an equivalence relation $ \sim $ on $ E $ by $ \alpha \sim \beta $ if and only if they coincide on some open subgroup of $ U $. We set $ \Comm(U) = E/\sim $.
\end{de}

As explained in \cite{BEW11}, the group of abstract commensurators of an open compact subgroup of a simple algebraic group is described by Corollary 0.3 of Pink's paper \cite{P98}. Let us record that result explicitly.

\begin{theorem}[Corollary 0.3 in \cite{P98}]
Let $ G $ (resp. $ G' $) be an absolutely simple, simply connected alegebraic group over a local field $ k $ (resp. $ k' $). Let $ U $ (resp. $ U' $) be an open compact subgroup of $ G $ (resp. $ G' $). Then for any topological isomorphism $ \alpha : U\rightarrow U' $, there exists a unique isomorphism of algebraic groups $ G\rightarrow G' $ over a unique isomorphism of topological fields $ k\rightarrow k' $ such that the induced morphism $ G(k)\rightarrow G(k') $ extends $ \alpha $.
\end{theorem}

Given a topological group $G$, we denote by $\mathrm{Aut}(G)$ its group of bi-continuous automorphisms. 

\begin{cor}\label{computecomm}
Let $ G $ be an absolutely simple, simply connected algebraic group over a local field $ k $, and let $ U $ be an open compact subgroup of $ G(k) $. Then $ \Comm(U) $ is canonically isomorphic to $ \mathrm{Aut}(G(k)) $.
\end{cor}

\begin{proof}
Indeed, an element of $ \Comm(U) $ uniquely extends to an automorphism of $ G(k) $ by the previous theorem. Conversely, any automorphism $ \alpha $ of $ G(k) $ comes from such an extension, simply by looking at the restriction of $ \alpha $ to $ U\cap \alpha^{-1}(U) $.
\end{proof}

A priori, the group $\Comm(U)$ is just an abstract group, but as discussed in \cite{BEW11}, there are several ways to endow it with a group topology. The identification provided by Corollary~\ref{computecomm} suggests that, in our situation, the natural topology on $\Comm(U)$ should be the one which coincides with the Braconnier topology on $\mathrm{Aut}(G(k))$. Let us now address the details, following Section~7 from \cite{BEW11}.

\begin{de}
A profinite group $ U $ is called \textbf{countably characteristically based} if it has a countable basis of neighbourhood of the identity consisting of characteristic subgroups. A profinite group is called \textbf{hereditarily countably characteristically based} (h.c.c.b.) if every open subgroup of $ U $ is countably characteristically based.
\end{de}

\begin{ex}\label{h.c.c.b.}
Let $ G $ be a $ k $-simple, simply connected algebraic $ k $-group, where $ k $ is a local field, and let $ U $ be an open compact subgroup of $ G(k) $. Then $ U $ is h.c.c.b. Indeed, $ U $ is a h.j.i. virtually pro-$ p $ group (see Theorem \ref{h.j.i.}), hence is finitely generated. And as explained in \cite{BEW11}, section 7.1, every finitely generated profinite group is h.c.c.b. Another way to see that $ U $ is h.c.c.b. is to exhibit by hand a countable characteristic neighbourhood basis of any open subgroup by considering intersection of maximal open normal subgroups. 
\end{ex}

\begin{de}
Let $ U $ be an h.c.c.b. profinite group. For any open subgroup $ V\leq U $, let $ \rho _{V} : \mathrm{Aut}(V)\rightarrow \Comm(U) $ be the natural homomorphism and endow $ \mathrm{Aut}(V) $ with the compact-open topology. The \textbf{Aut-topology} on $ \Comm(U) $ is defined by the following sub-base of identity neighbourhood : 
\begin{align*}
\mathcal{B}_{U} = \lbrace H\leq  \Comm(U)\: \vert \:   \rho_{V}^{-1}(H)\: \text{is open in}\: \mathrm{Aut}(V)
 \text{ for all open subgroups}\: V \: \text{of}\: U\rbrace. 
\end{align*}
As explained in \cite{BEW11}, Proposition 7.3, this turns $ \Comm(U) $ into a topological group.
\end{de}


\begin{de}
Let $ G $ be a locally compact group. The \textbf{Braconnier} topology on $ \mathrm{Aut}(G) $ is defined by the following sub-base of identity neighbourhood : 
\begin{align*}
\mathcal{U}(K,U) = \lbrace \phi \in \mathrm{Aut}(G)\: \vert \: \forall \: x\in K,\: \phi (x) \in xU\: \text{and}\: \phi^{-1} (x) \in xU \rbrace,
\end{align*}
 where $ K\subseteq G $ is compact and $ U\subseteq G $ is an identity neighbourhood (see e.g. \cite{CM11}, Appendix I for more comments on this topology).
\end{de}

\begin{rem}\label{continuityoftheadjointmap}
The Braconnier topology is the natural one, in the sense that it turns $\mathrm{Aut(G)}$ into a topological group, while the compact-open topology on   $\mathrm{Aut(G)}$ does not in general. However it does in the special case where $G$ is compact. Moreover, given any closed normal subgroup $N$ of $G$,  the adjoint map $ \mathrm{Ad} \colon G\rightarrow \mathrm{Aut}(N) $  given by the conjugation action is continuous for the Braconnier topology (see \cite{HR79}, Theorem (26.7)).
\end{rem}

In order to prove that the Aut-topology on $ \Comm(U) $ coincide with the Braconnier topology on $\mathrm{Aut}(G(k)) $, we use the following result due to Barnea--Ershov--Weigel. 

\begin{prop}\label{toponComm}
Let $ U $ be an h.c.c.b. profinite group such that $ \Comm(U) $ with the Aut-topology is Hausdorff. Suppose that $ \Comm(U) $ is a topological group with respect to some topology $ \mathcal{T} $ and that there exists an open subgroup $ V $ of $ U $ such that
\begin{enumerate}[(i)]
\item 
The index $[\Comm(U) : \mathrm{Aut}(V)]$ is countable.

\item $ \mathrm{Aut}(V) $ is an open compact subgroup of $(\Comm(U), \mathcal{T} )$.

\item If $ N $ is an open subgroup of $ V $ and $ (f_{n})_{n=1}^{\infty} $ is a sequence in $ \mathrm{Aut}(V) $ such that $ f_{n}\rightarrow 1 $ with respect to $ \mathcal{T} $, then $ f_{n}(N) = N $ for sufficiently large $n$.
\end{enumerate}

Then $(\Comm(U), \mathcal T)$ is locally compact, second countable,  and $ \mathcal{T} $ coincides with the Aut-topology.
\end{prop}
\begin{proof}
See Proposition 8.8 in \cite{BEW11} for the fact that $ \Comm(U) $ is a locally compact group. Since $U$ is h.c.c.b., it follows that $\mathrm{Aut}(V)$ is a compact metrisable group with respect to the compact-open topology. By definition of the Aut-topology, the natural embedding $\mathrm{Aut}(V)\to \Comm(U)$ is continuous, and is thus a homeomorphism onto its image. Therefore $\Comm(U)$ is metrisable. Moreover it is $\sigma$-compact since $[\Comm(U) : \mathrm{Aut}(V)]$ is countable. This confirms that $\Comm(U)$ is second countable.
\end{proof}

The following result is a straightforward adaption of Example~8.1 from \cite{BEW11} to our situation.

\begin{prop}\label{topiso}
Let $ H $ be an absolutely simple, simply connected algebraic group over a local field $ k $, and let $ U $ be an open compact subgroup of $ H(k) $. Then the canonical isomorphism $ \Comm(U)\simeq \mathrm{Aut}(H(k)) $  of Corollary \ref{computecomm} is an isomorphism of topological groups, where $ \Comm(U) $ has the $ \mathrm{Aut} $-topology and $ \mathrm{Aut}(H(k)) $ has the Braconnier topology. In particular  $ \mathrm{Aut}(H(k)) $ is a \tdlcsc group.
\end{prop}

\begin{proof}
As noted in Example \ref{h.c.c.b.},  $ U $ is h.c.c.b. We next claim that $ \Comm(U) $ is Hausdorff. Indeed, $ \Comm(U)=\Comm(V) $ for some open subgroup $ V $ having a trivial center, so that $ QZ(V) $ is trivial (to prove this last assertion, one can argue as in \cite{BEW11}, Proposition 5.1). Using \cite{BEW11}, Proposition 2.5, this implies the claim.

The desired conclusion will follow from Proposition \ref{toponComm}. In order to check the three conditions, we first remark that, since every automorphism of $ U $ extends to the whole of $ H(k) $, we have 
\begin{align}\label{eq}
\mathrm{Aut}(U) = \lbrace \varphi \in \mathrm{Aut}(H(k))\; \vert \; \varphi (U)= U \rbrace = \mathcal{U}(U,U).
\end{align}
We now check the three conditions successively. 

(i) The index of $ \mathrm{Aut}(U) $ in $ \Comm(U) $ is countable. Indeed, $ \phi $, $ \psi \in \Comm(U) $ are in the same  coset modulo $ \mathrm{Aut}(U) $ if and only if $ \phi (U)=\psi (U) $. Therefore it suffices to check that $U$ has a countable orbit under $\Comm(U) = \mathrm{Aut}(H(k))$. This is indeed the case, since $H(k)$ is second countable, and thus has countably many compact open subgroups. 

(ii) In view of (\ref{eq}), $ \mathrm{Aut}(U) $ is open in $ \mathrm{Aut}(H(k)) $ by the definition of the Braconnier topology.

(iii) Let $ N $ be an open subgroup of $ U $, and let $ (f_{n})_{n=1}^{\infty} $ be a sequence converging to $ 1 $ in $ \mathrm{Aut}(H(k)) $. Then, for $ n $ large enough, $ f_{n}\in \mathcal{U}(N,N) = \lbrace \varphi \in \mathrm{Aut}(H(k))$ $\vert $ $ \varphi (N) = N \rbrace $, as wanted.
\end{proof}

If $U$ is an open compact subgroup of $G$, we have a canonical map  $ G\rightarrow \Comm(U) $; 
we end this section by verifying its continuity.

\begin{lem}\label{generalcontinuity}
Let $ U $ be an h.c.c.b. profinite group and let $ G $ be a topological group containing $ U $ as a locally normal subgroup. Assume that $ G $ commensurates $ U $. Then the canonical map $ \varphi \colon G\rightarrow \Comm(U) $ is continuous, where $ \Comm(U) $ has the $ \mathrm{Aut} $-topology.
\end{lem}

\begin{proof}
Let $ W=N_{G}(U) $, which is open by assumption. It suffices to prove that the restriction of $ \varphi $ to $ W $ is continuous at the identity. Observe that $ \varphi $ factors through $ \rho_{U} \colon \mathrm{Aut}(U)\rightarrow \Comm(U)  $, which is continuous by definition of the $ \mathrm{Aut} $-topology. Moreover the adjoint map $W \to \mathrm{Aut}(U)$ is continous by Remark~\ref{continuityoftheadjointmap}, so that the composed map $W \to \mathrm{Aut}(U) \to \Comm(U)$ is continuous as well.
\end{proof}

\subsection{Decomposition into hereditarily just-infinite factors}

We now come to the heart of our toolbox, which consists of Pink's results from \cite{P98}. We start by repeating one of the main theorems from loc. cit.

\begin{theorem}\label{cor0.5}
Let $ k $ be a local field and let $ \Gamma $ be a compact subgroup of $ GL_{n}(k) $. There exist closed normal subgroups $ U_{3}\leq U_{2}\leq U_{1} $ of $ \Gamma $ such that

\begin{enumerate}[(i)]
\item $ U_{1} $ is of finite index in $ \Gamma $.
\item $ U_{1}/U_{2} $ is abelian of finite exponent.
\item There exists a local field $ k' $ of the same characteristic and residue characteristic as  $ k $, a Zariski-connected, semisimple adjoint algebraic group $ H $ over $ k' $, with universal covering $ \pi \colon \tilde{H}\rightarrow H $, and an open compact subgroup $ \Delta \leq \tilde{H}(k')  $ such that $ U_{2}/U_{3}\simeq \pi_{k'} (\Delta ) $ as topological groups.
\item $ U_{3} $ is solvable of derived length $  $ $ \leq n $.
\end{enumerate}

\end{theorem}
\begin{proof}
See \cite{P98}, Corollary 0.5.
\end{proof}

It will be crucial for our purposes to arrange that the subquotient $ U_{2}/U_{3} $ is the direct product of h.j.i. groups. This is achieved by the following. 

\begin{theorem}[Extended version of \cite{P98}, Corollary 0.5]\label{cor0.5bis}
Let $ k $ be a local field and let $ \Gamma $ be a compact subgroup of $ GL_{n}(k) $. There exist closed normal subgroups $ U_{3}\leq U_{2}\leq U_{1} $ of $ \Gamma $ such that :

\begin{enumerate}[(i)]
\item $ U_{1} $ is of finite index in $ \Gamma $.

\item $ U_{1}/U_{2} $ is abelian of finite exponent.

\item There exist local fields $ k'_{1},\dots,k'_{m} $ of the same characteristic and residue characteristic than  $ k $, Zariski-connected, absolutely simple adjoint algebraic $ k'_{i} $-group $ H_{i} $, with universal covering $ \pi_{i} \colon \tilde{H_{i}}\rightarrow H_{i} $, and open compact subgroups $ \Delta_{i} \leq \tilde{H_{i}}(k'_{i})  $ such that $ U_{2}/U_{3}\simeq \pi_{1} (\Delta_{1} )\times \dots \times \pi_{m} (\Delta_{m} )  $ as topological groups. In particular, the subquotient $ U_{2}/U_{3} $ is a direct product of non-virtually abelian h.j.i. groups.

\item $ U_{3} $ is solvable of derived length  $ \leq n $.

\end{enumerate}

\end{theorem}

\begin{proof}
Retain the notation of Theorem~\ref{cor0.5}. As recalled in \S \ref{sec:LAG}, we may decompose $ \tilde{H} $ as the direct product of Weil restrictions $ \prod \limits_{i=1}^{m} \mathcal{R}_{k_{i}'/k'} \tilde{H}_{i} $, where each $ \tilde{H}_{i} $ is an absolutely simple, simply connected algebraic group over a finite separable extension $ k_{i}' $ of $ k' $. Let $ G_{i} = \mathcal{R}_{k_{i}'/k'} \tilde{H}_{i}(k')  $.



Now, the compact group $ \Delta $ appearing in (iii) of Theorem \ref{cor0.5} is an open compact subgroup of $ G_{1}\times \dots \times G_{m} $. Therefore there  exists an open compact subgroup $\Delta_i$ for $G_i$ such that $\Lambda = \Delta_1 \times \dots \times \Delta_m$ is contained in $\Delta$. Thus $\pi(\Lambda) \simeq (\pi_{1})_{k'_{1}}(\Delta_1)\times \dots \times (\pi_{m})_{k'_{m}}(\Delta_m)  $ is an open subgroup of $ U_{2}/U_{3} $. However, it is not clear a priori that it is normalised by $\Gamma/U_3$. In order to ensure that, it suffices to apply Lemma~\ref{lem:ProductHJI} to the group $U_2/U_3$. This shows that, upon replacing each $\Delta_i$ by a suitable open subgroup, the image $\pi(\Lambda)$ is indeed an open subgroup of $U_2/U_3$ of the desired form, which is moreover normalised by $\Gamma/U_3$.

Let $U'_2$ denote the preimage of $\pi(\Lambda)$. Now the quotient $U_1/U'_2$ is finite-by-\{abelian of finite exponent\}. We may thus replace $U_1$ by a smaller open  normal subgroup $U'_1$ of $\Gamma$ containing $U'_2$ to ensure that $U'_1/U'_2$ is indeed abelian of finite exponent. Now the normal chain $U_3 \leq U'_2 \leq U'_1$ of $\Gamma$ satisfies all the requested properties. 
\end{proof}

To capture the properties of the compact factors appearing in  (iii) of Theorem~\ref{cor0.5bis}, we introduce the following terminology. 

\begin{de}\label{def:algebraicCompact}
A compact h.j.i. group $\Gamma$ is called \textbf{algebraic} if  there is a local field $k$ and a Zariski-connected, absolutely simple, adjoint algebraic $k$-group $H$ with universal cover $\pi \colon \tilde H \to H$, and a compact open subgroup $\Delta$ of $\tilde H(k)$ such that $\Gamma$ is isomorphic to $\pi(\Delta)$.
\end{de}

\section{The global structure of locally linear groups}

\subsection{Proof of the main theorem}

The following result is a reformulation of Theorem~\ref{thm:Main1} from the introduction, using the terminology introduced in  \S\ref{radicaltheories}.   

\begin{theorem}\label{mainthm}
Let $ G $ be a \tdlc group having an open compact subgroup which is linear over a local field $k$. Then $ G $ has a series of closed normal  subgroups: 
$$ \lbrace 1\rbrace \unlhd R \unlhd G_{1} \unlhd G_{0} \unlhd G $$ 
enjoying the following properties. 

The group $ R  = R_{[A]}(G)$ is the $ [A] $-regular radical of $ G $ and is locally solvable. The group $ G_{0} $ is an open characteristic subgroup of finite index in $ G $. The   quotient group $H_0 = G_0/R$, if non-trivial, has non-trivial closed normal subgroups, say $M_1, \dots, M_m$, satisfying the following.

\begin{enumerate}[(i)]

\item For some $l \leq m$ and all $i \leq l$, the group $M_i$ is a topologically simple algebraic group over a local field $k_i$, of the same characteristic and residue characteristic as $k$. In particular $M_i$ is compactly generated and abstractly simple. 

\item For all $j > l$, the group $M_j$ is   compact, h.j.i., and algebraic (in the sense of Definition~\ref{def:algebraicCompact}) over a local field $k_j$, of the same characteristic and residue characteristic as $k$.

\item Every non-trivial closed normal subgroup $N$ of $H_0$ contains $M_i$ for some $i \leq l$, or contains an open subgroup of $M_j$ for some $j > l$. 

\item The quotient group $H_{1} = G_1/R$ coincides with the product $  M_1 \dots M_m \cong M_1 \times \dots \times M_m$, which is closed in $H_0$.  In particular, $ H_{1} $ is compactly generated. Moreover  $H_{0}/H_{1} = G_0/G_{1}$ is  {locally abelian}.
\end{enumerate}
\end{theorem}

\begin{proof}
Let $U \leq G$ be a compact open subgroup which is linear over $k$. Let also $H = G/R$ and $V$ denote the image of $U$ in $H$. Theorem~\ref{cor0.5bis} applied  to the group $U$ yields closed normal subgroup $U_3 \leq U_2 \leq U_1$ satisfying the properties (i), (ii), (iii) and (iv) from that statement. 

\begin{claim}\label{claim:0}
$U_3$ is contained in $R$ as an open subgroup.  In particular $R$ is locally solvable.
\end{claim}

The image of $U_3$ in $H$ is a solvable locally normal subgroup. It must therefore be trivial, since $H$ is locally C-stable by Theorem~\ref{A-rad}. Thus $U_3 \leq R$. 

Assume now for a contradiction that $U_3$ is not open in $R$. Then $U \cap R$ contains $U_3$ as a closed normal subgroup of infinite index. Since $U \cap R$ is $[A]$-regular in $U$ by Theorem~\ref{A-rad}, it follows that the image of $U \cap R$ in $U/U_3$ contains a non-trivial locally normal subgroup belonging to $[A]$. However, by Lemma~\ref{structurelattice} and Theorem~\ref{cor0.5bis}, every non-trivial locally normal subgroup of $U/U_3$ contains a locally normal subgroup which is h.j.i. and algebraic. Those subgroups do not belong to $[A]$. This is a contradiction, and the claim stands proven.

\begin{claim}\label{claim:1}
There exist closed normal subgroups $V_2 \leq V_1$ of $V$ such that 
\begin{enumerate}[(i)]
\item $ V_{1} $ is of finite index in $ V $.

\item $ V_{1}/V_{2} $ is abelian of finite exponent.

\item There exist local fields $ k'_{1},\dots,k'_{m} $ of the same characteristic and residue characteristic than  $ k $, Zariski-connected, absolutely simple adjoint algebraic $ k'_{i} $-group $ H_{i} $, with universal covering $ \pi_{i} \colon \tilde{H_{i}}\rightarrow H_{i} $, and open compact subgroups $ \Delta_{i} \leq \tilde{H_{i}}(k'_{i})  $ such that $ V_{2} \simeq \pi_{1} (\Delta_{1} )\times \dots \times \pi_{m} (\Delta_{m} )  $ as topological groups.    In particular, the group $ V_{2}$ is a direct product of non-virtually abelian h.j.i. groups.
\end{enumerate}
\end{claim}

We denote by $V_i$ the image of $U_i$ in $V$. Then $V_i$ is a closed normal subgroup of $V$ (because $U_i$ is compact), and   $V_3$ is trivial by Claim~\ref{claim:0}. 

Notice that $V_1$ and $V_2$ satisfy conditions (i) and (ii) from the claim, in view of the corresponding properties of $U_1$ and $U_2$. It remains to check that $V_2$ satisfies (iii). Since $V_2 \simeq U_{2}R/R\simeq U_{2}/U_{2}\cap R$, it suffices to show that $U_2 \cap R$ is trivial in view of Theorem~\ref{cor0.5bis}(iii). By Claim~\ref{claim:0}, the group $U_3 $ is contained as an open subgroup of $U_2 \cap R$, so that the image of $U_2 \cap R$ in $U_2/U_3$ is a finite normal subgroup, and is thus trivial by Theorem~\ref{cor0.5bis}(iii) and Lemma~\ref{2m}. This shows that $U_2 \cap R = U_3$, so that $V_2 \simeq U_2/U_3$. 
The claim stands proven. 

\begin{claim}\label{claim:2}
The set of atoms of $\mathcal{LN}(H)$ coincides with $ \lbrace [\pi_{1}(\Delta_{1})],\dots,[\pi_{m}(\Delta_{m})]\rbrace $, and every non-zero element of $\mathcal{LN}(H)$ contains an atom. In particular $H$ has an open characteristic subgroup $H_0$ containing $V_2$, which commensurates $\pi_i(\Delta_i)$ for all $i = 1, \dots, m$.
\end{claim}

Since $V_1$ is open in $H$, we have $\mathcal{LN}(H) = \mathcal{LN}(V_1)$. In view of Claim~\ref{claim:1}, the hypotheses of Lemma~\ref{structurelattice} are satisfied by $V_1$. This proves the desired assertions on $\mathcal{LN}(H)$. 

Now the $H$-action on $\mathcal{LN}(H)$ permutes the atoms, and we define $H_0$ to be the kernel of that permutation action. Then $H_0$ is indeed open, characteristic and of finite index in $H$, and commensurates $\pi_i(\Delta_i)$ for all $i$. Since $V_2$ normalises $\pi_{i}(\Delta_{i})$ for all $i$, we have $V_2 \leq H_0$, as claimed. 

\begin{claim}\label{claim:3}
For each $i \in \{1, \dots, m\}$, let $\varphi_i \colon H_{0}\rightarrow \Comm(\pi_{i}(\Delta_{i}))$ be the homomorphism  induced by Claim~\ref{claim:2}. Then the  product homomorphism 
$$ 
\varphi = \varphi_1 \times \dots \times \varphi_m \colon H_{0}\rightarrow \Comm(\pi_{1}(\Delta_{1}))\times \dots\times \Comm(\pi_{m}(\Delta_{m})) 
$$  
is  continuous and injective, where each factor is endowed with the Aut-topology.
\end{claim}

In view of Lemma~\ref{generalcontinuity},  the map $ \varphi $ is a product of continuous homomorphisms, and  is thus continuous. Let us now check its injectivity.

In view of Definition~\ref{Cstable}, we have $ \Ker \varphi \leq QC_{H_0}(V_{2})$, and it suffices to check that $V_2$ has trivial quasi-centraliser in $H_0$. 

Recalling that $H$, and thus also $H_0$, is locally C-stable, we deduce from Lemma~3.9 and Proposition~3.14 from \cite{CRW13} that  $ QC_{H_0}(V_{2}C_{H_0}(V_{2})) =1$. Since $V_1$ is open in $H$,  we have $1 = QC_{H_0}(V_{2}C_{H_0}(V_{2})) = QC_{H_0}((V_{2}C_{H_0}(V_{2})) \cap V_1) = QC_{H_0}(V_{2}C_{V_1}(V_{2}) )$. Therefore it is enough to show that the centraliser $C_{V_1}(V_2)$ is trivial. 
Now  $ C_{V_{1}}(V_{2})\cap V_{2} \leq  QZ(V_{2}) $, which is trivial by Proposition~\ref{prop:locCstable}. Thus, $ C_{V_{1}}(V_{2}) $ embeds into $ V_{1}/V_{2} $, and  is thus abelian by Claim~\ref{claim:1}. But $ C_{V_{1}}(V_{2}) $ is also locally normal in $H$ (see e.g. \cite{CRW13}, Lemma 2.1), and must therefore be trivial because  $H$ is locally C-stable.

\begin{claim}\label{claim:4}
Let $i \in \{1, \dots, m\}$. Then there is an isomorphism of topological groups 
$$\Comm(\pi_i(\Delta_i)) \simeq \Aut(\tilde H_i(k'_i)),$$ 
where $\Comm(\pi_i(\Delta_i))$ has the Aut-topology and $\Aut(\tilde H_i(k'_i))$ has the Braconnier topology. 
\end{claim}

By Corollary~\ref{computecomm} and Proposition~\ref{toponComm}, we have an isomorphism of topological groups $\Comm(\Delta_i) \simeq \Aut(\tilde H_i(k'_i))$. Since $\Ker \pi_i$ is finite, there is an open subgroup $\Delta_i' \leq \Delta_i$ which intersects $\Ker \pi_i$ trivially. Thus $\pi_i$ induces an isomorphism of profinite groups between $\Delta'_i$ and its image, so that 
$$\Comm(\pi_i(\Delta_i)) = \Comm(\pi_i(\Delta'_i)) \simeq \Comm(\Delta'_i) = \Comm(\Delta_i).$$
The claim follows. 

\begin{claim}\label{claim:5}
Let $i \in \{1, \dots, m\}$ and set 
$$ M_{i} = \varphi_{i}^{-1}(\Inn(\tilde{H_{i}}(k_{i}')))\cap \bigcap \limits_{j\neq i} \Ker \varphi_{j},$$ 
where $\Inn(\tilde{H_{i}}(k_{i}'))$ is viewed as a subgroup of $\Comm(\pi_i(\Delta_i))$ by means of Claim~\ref{claim:4}. Then $M_i$ is a closed normal subgroup of $H_0$, and exactly one of the following assertions holds:
\begin{enumerate}[(a)]
\item $M_i$ is a compact, h.j.i. group which is algebraic over $k'_i$. 

\item $M_i \simeq \tilde H_i(k'_i)/Z(\tilde H_i(k'_i))$, and $\tilde H_i$ is isotropic over $k'_i$. In particular $M_i$ is a topologically simple algebraic group over $k'_i$. 
 
\end{enumerate}
\end{claim}

We first check that the quotient group $\tilde H_i(k'_i)/Z(\tilde H_i(k'_i))$ is isomorphic to $\Inn(\tilde{H_{i}}(k_{i}'))$ endowed with the Braconnier topology. Indeed, by Proposition~\ref{topiso}, the group $\Aut(\tilde{H_{i}}(k_{i}'))$ is locally compact, and by Remark~\ref{continuityoftheadjointmap} the canonical embedding $\tilde H_i(k'_i)/Z(\tilde H_i(k'_i)) \to \Aut(\tilde{H_{i}}(k_{i}'))$ is continuous. From Proposition~\ref{closed image}, we deduce that the latter embedding is a homeomorphism onto its image, namely $\Inn(\tilde{H_{i}}(k_{i}'))$, and that the latter is closed in $\Aut(\tilde{H_{i}}(k_{i}'))$.  This also implies that $M_i$ is a closed normal subgroup of $H_0$. 

We next observe that the restriction of  $\varphi_i$  to $M_i$ is a homeomorphism onto its image. Indeed $(\varphi_i)|_{M_i}$ is injective by Claim~\ref{claim:3}. Moreover  by Claim~\ref{claim:1}, we have $\pi_i(\Delta_i) \leq M_i$, and $ \varphi _{i}(\pi_{i}(\Delta_{i})) $ is open in $ \mathrm{Inn}(\tilde{H_{i}}(k_{i}'))\simeq \tilde{H_{i}}(k_{i}')/Z(\tilde{H_{i}}(k_{i}')) $. Thus $(\varphi_i)|_{M_i}$ is a continuous isomorphism onto an open, hence closed, subgroup of $ \mathrm{Inn}(\tilde{H_{i}}(k_{i}'))$. 

Since $ \mathrm{Inn}(\tilde{H_{i}}(k_{i}'))$ is second countable (see Poposition~\ref{topiso}) and $(\varphi_i)|_{M_i}$ is injective, we deduce that the compact group $\pi_i(\Delta_i)$ is of countable index in $M_i$. It follows that $M_i$ is $\sigma$-compact. By the Open Map Theorem  (see \cite{HR79}, Theorem (5.29)), we deduce that the map $(\varphi_i)|_{M_i}$ is open, as requested. 

Now if $M_i$ is compact, the desired claim follows by construction (see Theorem~\ref{h.j.i.} and Definition~\ref{def:algebraicCompact}). If $M_i$ is non-compact, then $\varphi_i(M_i)$ is a non-compact open subgroup  of $\Inn(\tilde H_i(k'_i))$ so that $\tilde H_i(k'_i)$ is non-compact. Hence $\tilde H_i$ is isotropic by \cite{P82}.  By Theorem~\ref{thmT}, the only non-compact open subgroup of $ \mathrm{Inn}(\tilde{H_{i}}(k_{i}'))\simeq \tilde{H_{i}}(k_{i}')/Z(\tilde{H_{i}}(k_{i}')) =  \tilde{H_{i}}(k_{i}')^+/Z(\tilde{H_{i}}(k_{i}')^+)$ is the whole group (see Theorem~\ref{thm:BT} for the last equality). The claim follows. 

\begin{claim}\label{claim:6}
We have $[M_i, M_j]=1$ for $i \neq j$. Moreover, the product $H_{1} = M_1 \dots M_m \cong M_1 \times \dots \times M_m$ is closed in $H_0$, and the quotient $H_0/H_{1}$ is  {locally abelian}.
\end{claim}

The injectivity of $\varphi$, established in Claim~\ref{claim:3}, ensures that the $M_i$'s commute pairwise, and that the canonical map from $M_1 \times \dots \times M_m$ onto the subgroup $S = M_1 \dots M_m$  is  a continuous isomorphism. To see that $H_{1}$ is closed, consider the canonical projection 
$H_0 \to H_0 / \overline{M_2 \dots M_m}$. If $M_1$ is compact, then it has closed image. If $M_1$ is not compact, then Claim~\ref{claim:5} and Proposition~\ref{closed image} ensure that $M_1$ has closed image as well. Hence the product $M_1\overline{M_2 \dots M_m}$ is closed in $H_0$, and a straightforward induction shows that $H_{1}$ is closed as well. 

Finally, since $V_2 \leq H_{1}$, it follows from Claim~\ref{claim:1} that $H_0/H_{1}$ is locally abelian.

\begin{claim}\label{claim:7}
Every non-trivial closed normal subgroup $N$ of $H_0$ contains some non-compact $M_i$, or an open subgroup of some compact  $M_j$.
\end{claim}

The group $V \cap N$ is a locally normal subgroup of $H_0$, and therefore there is some $i$ such that $[\pi_i(\Delta_i)] \leq [V \cap N]$ by Claim~\ref{claim:2}. If $M_i$ is compact, this yields the desired assertion. Otherwise we see that $N \cap M_i$ is an open normal subgroup of $M_i$, so that $M_i \leq N$ by Claim~\ref{claim:5}. The claim stands proven

\medskip
To conclude the proof, we denote by $G_0$ (resp. $ G_{1} $) the preimage of $H_0$ (resp. $ H_{1} $) in $G$, and re-order the set $\{M_1, \dots, M_m\}$ so that the non-compact elements come first. We see that all the requested assertions have been established in the claims above: Assertions~(v) in Claim~\ref{claim:6}, Assertion~(i) and (ii) in Claim~\ref{claim:5} and Assertion~(iii) in Claim~\ref{claim:7}.
\end{proof}

\subsection{Corollaries}

\begin{proof}[Proof of Corollary~\ref{cor:Main2}]
We apply Theorem~\ref{mainthm}, which yields subgroups  $M_i$  and $G_0$ of $G$. Let  $A$ be the   elementary radical of $G_0$, see Theorem~\ref{starelementaryradicalexistence}. 

\setcounter{claim}{0}
\begin{claim}\label{claim:1'}
We have $ R \leq A$. Moreover $A/R$ coincides with the   elementary radical of $H_0 = G_0/R$.
\end{claim}

By Theorem~\ref{thm:starradicals}, we have $ R \leq A$ and $R$ is elementary. Therefore $A/R$ contains the elementary radical of $H_0$. The claim follows, since the quotient group $A/R$ is elementary by Theorem~4.6(c) from \cite{W14c}.

\begin{claim}\label{claim:2'}
Set  $ W = \bigcap_{i=1}^{l} C_{H_{0}}(M_{i}) \leq H_0$, where $l$ is as in Theorem~\ref{mainthm}. Then $W$ is compact-by-\{locally abelian\}. In particular it is   elementary. 
\end{claim}

We have $M_j \leq W$ for all $j >l$ by Theorem~\ref{mainthm}(i). Thus $\tilde W = M_{l+1} \dots M_m$ is a compact normal subgroup of $W$. Moreover $W \cap (M_1 \dots M_l)= 1$, since the latter product has trivial center in view of  Theorem~\ref{mainthm}(ii) and (v). It follows that $W/\tilde W$ embeds into $H_0/H_{1}$, which is locally abelian. This implies that $W$ is elementary by Proposition~\ref{prop:transferfroms.c.togeneral}(ii). 

\begin{claim}\label{claim:3'}
Every non-trivial closed normal subgroup $N$ of $H_0$ which is not contained in $W$ contains some $M_i$ with $i\in \{1, \dots, l\}$. 
\end{claim}

Assume that $N$ does not contain any non-compact $M_i$. Then $[N, M_i] \leq N \cap M_i =1$ since $M_i$ is topologically simple. Thus $N \leq W$ as desired. 

\begin{claim}\label{claim:4'}
$W$ coincides with the   elementary radical of $H_0$. 
\end{claim}

That $W$ is contained in the   elementary radical follows from Claim~\ref{claim:2'}. If that inclusion were proper, then the   elementary radical of $H_0$ would contain some non-compact $M_i$ by Claim~\ref{claim:3'}. This is impossible because every closed subgroup of an   elementary group is   elementary (see \cite{W14c}, Theorem~4.6(b)), while  non-discrete compactly generated topologically simple groups are not elementary (see \cite{W14}, Proposition~6.2).

\medskip
To conclude the proof, we remark that $H' = H_0/W$ is isomorphic to $G_0/A$ in view of Claims~\ref{claim:1'} and~\ref{claim:4'}. Thus it suffices to show all the desired assertions for the quotient $H_0/W$. For each $i \in \{1, \dots, l\}$, we define a group $M'_i$ as the image of $M_i$ in the quotient $H'=H_0/W$. That image is injective because $M_i$ is simple, and closed by Proposition~\ref{closed image}. Thus each $M'_i$ is a topologically simple algebraic group over $k'_i$. The assertions that the $M'_i$ are precisely the minimal normal subgroups of $H_0/W$, and that every non-trivial closed normal subgroup contains one of them, follow from Claim~\ref{claim:3'}. That $H_{1}' = M'_1 \dots M'_l$ is closed follows from the same argument as in the proof of Claim~\ref{claim:6} in Theorem~\ref{mainthm}. Finally, that $(H_0/W)/H_{1}'$ is locally abelian follows from Theorem~\ref{mainthm}(v). Therefore that quotient is   elementary by  Proposition~\ref{prop:transferfroms.c.togeneral}(ii). 
\end{proof}

\begin{proof}[Proof of Corollary~\ref{cor:TopolSimple}]
Assume that $G$ is non-discrete.  Since it is topologically simple, its $ [A] $-regular radical $R$ is either trivial or the whole of $ G $. 

Assume that $R=G$. Then $G$ is locally solvable by Theorem~\ref{mainthm}. Moreover Theorem~5.3 of \cite{CRW14} implies that $G$ is not compactly generated. 
  
Assume now that $ R =1$. By the definition of the $ [A] $-regular radical, $ G $ is not locally abelian. Hence the product $ M_{1}\times \dots\times M_{m} $ from Theorem~\ref{mainthm} is non-trivial. Since $ G $ is topologically simple, we have $ m=1 $ and $ G= M_{1} $. Since $ G $ is not compact (because a topologically simple profinite group is finite, hence discrete), we obtain the desired conclusion. 
\end{proof}

\begin{proof}[Proof of Corollary~\ref{cor:S}]
Immediate from Corollary~\ref{cor:TopolSimple}.
\end{proof}

\begin{proof}[Proof of Corollary~\ref{cor:LC}]
Each class of groups listed in the statement consists of  linear groups. Assume conversely that $G$ is a compactly generated, topologically simple, locally compact group that is linear. If $G$ is connected, then it is a simple Lie group, as a consequence of the solution to Hilbert's fifth problem. Otherwise $G$ is totally disconnected. If it is non-discrete, then it is algebraic by Corollary~\ref{cor:S}. If it is discrete, then it is residually finite by a theorem of Mal'cev (see e.g. \cite{LS01} Window 7, Proposition 8), hence a finite simple group. 
\end{proof}

\subsection{Some examples}

In this section, we describe a family of examples of \tdlc groups satisfying the hypotheses of Theorem~\ref{thm:Main1}, and illustrating some peculiar properties   that the quotient $H_0= G_0/R$ can have in general. For the construction, we use the Nottingham group.

\begin{de}
The \textbf{Nottingham group}, denoted $ J $, is the group of normalized continuous automorphisms of the ring $ \mathbb{F}_{p}[\![T]\!] $. Otherwise stated, an element $ g\in J $ is defined by its action on $ T $ and is of the following form
$$ g(T) = T + \sum \limits_{i=2}^{\infty}a_{i}T^{i}, \: \: \; a_{i}\in \mathbb{F}_{p} $$
\end{de}

We will use the universality of the Nottingham group, asserted in the following.

\begin{theorem}[Main result in \cite{C97}]\label{universalityofNottingham}
Every countably based pro-$ p $ group can be embedded, as a closed subgroup, in the Nottingham group.
\end{theorem}

The following construction  shows that the group $H_0$ from Theorem~\ref{thm:Main1} need not be second countable, and that it need not have any maximal compact normal subgroup. 

\begin{ex}\label{ex}
Consider the algebraic group $ SL_{n} $ over the local field $ \mathbb{F}_{p}(\!(T)\!) $. Then $ U = SL_{n}(\mathbb{F}_{p}[\![T]\!])/Z(SL_{n}(\mathbb{F}_{p}[\![T]\!]))  $ is a compact linear group which is h.j.i. by Theorem~\ref{h.j.i.}. Let $ L $ be a \tdlc  group admitting a continuous embedding into  the Nottingham group $ J $. Then the semi-direct product $ G = U\rtimes L $ is a \tdlc group.\\

We claim that $ G $ is locally $ C $-stable. Let us first check that $ QZ(G) $ is trivial. First observe that   $ QZ(U) $ is trivial by \cite{BEW11}, Proposition 5.1. Hence, if $ ul\in QZ(U\rtimes L) $, then $ l\in L $ must be non-trivial. Since a non-trivial element in $ J $ acts by outer automorphism on $ U $, we deduce that $ QZ(G) $ is trivial.

We now show that $ G $ has no non-trivial locally normal abelian subgroup. Arguing by contradiction, let $ K $ be such a subgroup. Since $ QZ(G) $ is trivial, $ K $ must be infinite. Since $ U $ is h.j.i. but not virtually abelian, the intersection  $ U\cap K $ must be trivial. Thus $K$ commutes with $U$. This is impossible, because $L$ acts on $U$ by outer automorphisms, so that $C_G(U)=1$. This confirms the claim. 

The claim implies that $R=1$, and that $G=G_0=H_0$ in the notation of Theorem~\ref{thm:Main1}. We now specialise this family of examples in two ways. 

Taking $ L = J $, endowed with the discrete topology, we see that  $G$ is a metrisable, locally linear, \tdlc group which is not second countable. 

Now consider $ L = \bigoplus _{n\in \mathbf{N}} \mathbf{Z}/p^{n}\mathbf{Z} $, with the discrete topology. It embeds in the pro-$p$ group $ \prod_{n\in \mathbf{N}} \mathbf{Z}/p^{n}\mathbf{Z} $, which itself embeds in $J$ by Theorem \ref{universalityofNottingham}. 
In this situation, we see that $G$ is a locally linear, \tdlcsc  group, but has no maximal compact normal subgroup.
\end{ex}

\end{document}